\documentclass{ws-jktr}

\usepackage{lipsum} %insert picture on same page as text (paragraph)
\usepackage{picins} %insert picture on same page as text (paragraph)
\usepackage{graphicx} %insert picture
\usepackage{color}
\usepackage{mathrsfs}  %出现数学希腊符号 %出现数学希腊符号$\mathscr{OS}$
\usepackage{picinpar}

\newcommand{\JW}[1]{f_{#1}}
\newcommand{\coeff}[2]{\coefficient_{\in \JW{#1}^{#2}}}
\DeclareMathOperator*{\coefficient}{coeff}

\begin{document}

\title
{Projectors in the Virtual Temperley-Lieb Algebra}

\author{Qingying Deng$^{\dag}$, Xian'an Jin$^{\dag}$\thanks{Corresponding author}, Louis.H.Kauffman$^{\ddag}$}

\address{$^{\dag}$School of Mathematics and Computational Science, Xiangtan University, Xiangtan, Hunan 411105,
P. R. China \\
qingying@xtu.edu.cn(Qingying Deng)}

\address{$^{\dag}$School of Mathematical Sciences,\\ Xiamen University, Xiamen, Fujian 361005, \\P. R. China \\
xajin@xmu.edu.cn(Xian'an Jin)}

\address{$^{\ddag}$Department of Mathematics, Statistics, and Computer Science,\\ University of Illinois at Chicago,  Chicago\\USA \\
and\\
Department of Mechanics and Mathematics,\\
Novosibirsk State University, Novosibirsk\\Russia\\
kauffman@uic.edu(Louis H. Kauffman)}

\maketitle

\begin{abstract}
We present a method of defining projectors in the virtual Temperley-Lieb algebra, that generalizes the Jones-Wenzl projectors in Temperley-Lieb algebra.
We show that the projectors have similar properties with the Jones-Wenzl projectors, and contain an extra property which is associated with the virtual generator elements, that is, the product of a projector with a virtual generator is unchanged.
We also show the uniqueness of the projector $f_n$ in terms of its axiomatic properties in the virtual Temperley-Lieba algebra $VTL_n(d)$.
Finally, we find the coefficients of $f_n$ and give an explicit formula for the projector $f_n$.
\end{abstract}

$\mathbf{keywords:}$ virtual Temperley-Lieb algebra; projector; recurrence formula; coefficient.

\vskip0.5cm

\section{Introduction}
Wenzl \cite{Wenzl} and Jones \cite{Jones} had proposed the notion of projectors in Temperley-Lieb algebra.
Later, people called them \textit{Jones-Wenzl projectors}.
By using elementary linear algebra, Lickorish \cite{Lick} had independently proved some properties of Jones-Wenzl projectors, which are useful machineries for producing an invariant of 3-manifolds.
Kauffman and Lomonaco \cite{KL07} described the background for topological quantum computing in terms of Temperley-Lieb recoupling theory.
Kauffman and Lins discussed the Temperley-Lieb recoupling theory which is used to construct invariants of 3-manifolds in a monograph \cite{KauLins}.

In the literature, the virtual Temperley-Lieb ($VTL$) algebra has been implicitly argued \cite{DK05,KauKD}.
Dye and Kauffman \cite{DK05} had defined the Witten-Reshetikhin-Turaev invariant of virtual link diagrams, whose computation involves the $VTL$ algebra.
They tried to understand the structure of the projectors in the context of the $VTL$ algebra.
And they hoped that the results in $VTL$ algebra would help in understanding the extension of the Witten-Reshetikhin-Turaev invariant to virtual knot theory.
The $VTL$ algebra is an extension of the Temperly-Lieb ($TL$) algebra by
virtual elements, and it is isomorphic to the Brauer algebra \cite{Benkart,Brauer}.
For the diagram description of Brauer algebra, we refer the reader to \cite{Brauer,Liu}.
In $VTL$ algebra, we look at all connections using $n$ points on top and $n$ points on bottom with regard to two parallel rows of points in the plane.
We allow connections top to top, top to bottom, bottom to bottom.
Crossings of the connection arcs are all virtual crossings.
The $VTL$ algebra generated by extending $n$-tangle diagrams by including virtual
crossings is a diagrammatic version of the Brauer algebra.

The $VTL$ algebra was independently proposed from a pure algebraic viewpoint \cite{ZKW}.
As an extension of \cite{ZKW}, Zhang \emph{et al.} \cite{ZKG} studied the $VTL$ algebra such that the virtual braid group can be represented in the $VTL$ algebra,
and also discussed the relationship between the $VTL$ algebra and the Brauer algebra.
Mcphail-Snyder and Miller \cite{SnyderMiller} also wrote that Temperley-Lieb algebra leads to a ``virtual Temperley-Lieb category'' when we consider diagrams drawn on surfaces instead of plane.
Mcphail-Snyder and Miller \cite{SnyderMiller} also calculated out the formula of second projector in $VTL$, which be used to construct a polynomial invariant for graphs in surfaces.
It is well-known that every (oriented) virtual link can be represented by a virtual braid, whose closure is isotopic to the original link \cite{Kamada,KauLam,LLL,Manturov}.
Li \emph{et al.} \cite{LLL} studied the properties of the $VTL$ algebra and showed how the $f$-polynomial of a virtual knot can be derived from a representation of the virtual braid group into the $VTL$ algebra, which is an approach similar to Jones's original construction.
These are the applications of the $VTL$ algebra for the bracket polynomial and for representations of the virtual braid group.

Frenkel and Khovanov \cite{FK97} presented a method of calculating the coefficients appearing in the Jones-Wenzl projections in the $TL$ algebra. Morrison \cite{Mor} wrote a self-contained paper which also gave a method of calculating the coefficients.
We refer to more papers \cite{Ocn02,Rez02,Rez07} for interested reader.

In this paper, we review some basis concepts and give the definition of projectors of the virtual Temperley-Lieb algebra
%and give a direct diagrammatic method for the factorization of any connection element in virtual Temperley-Lieb algebra
in Section 1.
In Section 2, we construct the projectors of the virtual Temperley-Lieb algebra recursively and demonstrate that the uniqueness of the projectors that satisfies all axiomatic.
In Section 3, we simplify the recurrence formula for the projectors by introducing a special linear span of diagrams.
In Section 4, we deduce some results on the coefficients appearing in the projectors in the virtual Temperley-Lieb algebra.
In Section 5, we give an explicit formula for the projectors in the virtual Temperley-Lieb algebra.

\section{The projectors in virtual Temperley-Lieb algebra}
\noindent

%A \emph{virtual $n$-tangle diagram} is a decorated immersion of n copies of $[0,1]$ with classical and virtual crossings. The $2n$ endpoints of the diagram are arranged so that $n$ endpoints appear in a row at the top and $n$ endpoints form a lower row. Any two virtual $n$-tangle diagrams can be multiplied
%by attaching the lower $n$ endpoints of one $n$-tangle to the top $n$ endpoints of another $n$-tangle.

% We define an $n$-tangle to be \emph{elementary} if it contains no classical or virtual crossings. Note that the product of any two elementary tangles is elementary. The \emph{closure} of an $n$-tangle connects each upper endpoint with a corresponding lower endpoint.
We first recall some elementary algebra notions involved in this article.
An (unital) associative algebra $\mathcal{A}$ \cite{Assem} is an algebra structure with compatible operation of addition, multiplication (associative), and a scalar multiplication by elements in some field $K$.
The addition and multiplication operations together give $\mathcal{A}$ the structure of a ring (with an identity element).
The addition and scalar multiplication operations together give $\mathcal{A}$ the structure of a vector space over $K$.
A $K$-vector subspace $\mathcal{B}$ of a $K$-algebra $\mathcal{A}$ is a $K$-subalgebra of $\mathcal{A}$ if the identity of $\mathcal{A}$ belongs to $\mathcal{B}$ and $bb'\in \mathcal{B}$ for all $b, b'\in\mathcal{B}$.

We assume that the reader is familiar with Reidemeister moves, virtual Reidemeister moves, virtual isotopy etc \cite{VKT}. The \emph{$n$-strand virtual Temperley-Lieb algebra} $VTL_{n}(d)$ is an associated algebra (with unit) over $\mathbb{C}(d)$ ($d$ is an indeterminate), the field of rational functions (field of fractions). Every element of this algebra is a linear combination of some 1-dimensional sub-manifolds (connection elements).
Each sub-manifold intersects the top and bottom of the rectangle in $n$ points.
If two sub-manifolds are virtual isotopic, keeping the boundary fixed, then we call them \emph{equivalent}.
Deleting a simple closed curve is equivalent to generating a multiple of $d$.
The product between two connection elements is represented by a vertical stack of rectangles attaching the bottom of one rectangle to the top of the other.
That is, the diagrammatic formulation of $VTL_n(d)$ is equivalent to Brauer algebra \cite{Benkart}.
The dimensions of $VTL_{n}(d)$ is $(2n-1)!!$.

Let $G_{n}=\{1_{n},e_1,e_2,...,e_{n-1},v_1,v_2,...,v_{n-1}\}$ be a set of $VTL_{n}(d)$ as illustrated in Figure \ref{F:VTLAgenerators}. In \cite{Cohen,Liu}, the authors usually call $v_{i}s$ Coxeter generators, and $e_{i}s$ Temperley-Lieb generators. In this article, we call $v_{i}s$ and $e_{i}s$ generators for simplicity.
%$G_n$ is called as the generator set of $VTL_{n}(d)$ in \cite{LLL}.
\begin{figure}[!htbp]
\centering
\includegraphics[width=12cm]{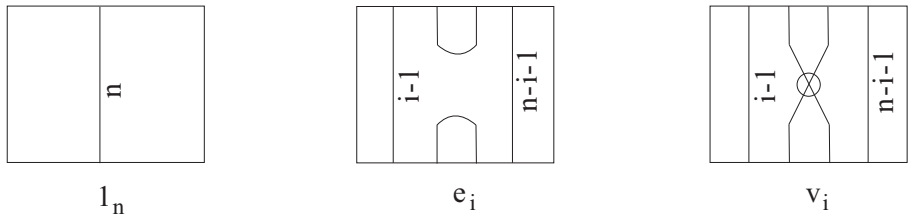}
 \caption{A set $G_{n}$ of $VTL_{n}(d)$.}\label{F:VTLAgenerators}
\end{figure}

For any $i,j=1,2,...,n-1$, the elements of $G_{n}$ satisfy the following relations:
\begin{equation}\label{vrelation}
\begin{split}
  &1_{n}e_i=e_i=e_i1_{n} \\
    &e_i^2=d e_i\\
    &e_ie_j=e_je_i \ \ if  \ |i-j|\geq 2\\
       &e_ie_{i\pm1}e_i=e_{i}\\
       &e_iv_{i\pm1}e_i=e_{i}\\
       &v_i^2=1_n\\
        &v_iv_j=v_jv_i \ \ if  \ |i-j|\geq 2\\
        &v_iv_{i+1}v_i=v_{i+1}v_iv_{i+1} \\
    &e_{i}v_j=v_je_{i} \ \ if \ |i-j|\geq 2\\
    &v_ie_{i+1}v_i=v_{i+1}e_iv_{i+1} \\
    &e_{i}v_i=v_ie_{i}=e_i\\
  \end{split}
\end{equation}

Note that we add a new formula $e_iv_{i\pm1}e_i=e_{i}$ which is not contained in \cite{LLL} in above relations such that the first pure virtual Reidemeister move appears in the diagrammatic representation.
But this formula is hold in \cite{Cohen,Liu}.
As far as the author is aware, the only complete proof of the diagrammatic formulation of $TL_n(d)$ is equivalent to the pure algebraic formulation of $TL_n(d)$ is the one given by Borisavljevi\'{c}, Do\v{s}en and Petri\'{c} in \cite{Borisavljevic}.
It is obvious that the diagrammatic formulation of $VTL_n(d)$ (which is equivalent to Brauer algebra) is equivalent to the pure algebraic formulation of $VTL_n(d)$.

In particular, the Temperley-Lieb algebra $TL_n(d)$ is a sub-algebra of $VTL_{n}(d)$, which does not involve $\{v_1,v_2,...,v_{n-1}\}$.
The dimensions of $TL_{n}(d)$ is $C_n=\frac{1}{n+1}\left(
                                                                           \begin{array}{c}
                                                                             2n \\
                                                                             n \\
                                                                           \end{array}
                                                                         \right)$
.

In Chapter 6 of \cite{Menasco}, Kauffman gave a direct diagrammatic method for expressing any connection element (1-dimensional sub-manifold) of $TL_{n}(d)$ as a canonical product of elements of $\{1_{n},e_1,e_2,...,e_{n-1}\}$.
Similarly, each connection element (1-dimensional sub-manifold) in $VTL_{n}(d)$ can be expressed as the product of elements of $G_n$.
Note that the expression is not unique.

Following Wenzl \cite{Wenzl} and Jones \cite{Jones} and guided also by Lickorish \cite{Lick} and \cite{KauLins}, consider for each $i\geq 1$, the polynomial function $\Delta_i$ {of degree} $i$ in $d$ defined recursively by
\begin{eqnarray}\label{E:Delta}
 \Delta_{i}&=& d \Delta_{i-1}-\Delta_{i-2}, \Delta_0=1~ and~ \Delta_{-1}=0.
\end{eqnarray}
This $\Delta_i$ is, in fact, the $i^{th}$ renormalised Chebyshev polynomial of the second kind.
For each $i\in \mathbb{N}$, provided $d$ is chosen in $\mathbb{C}$ so that $\Delta_i\neq 0.$

Denote Jones-Wenzl projectors of $TL_{n}(d)$ by $P_i$ ($i\leq n$). Next we shall review the combinatorial definition and the properties of $P_i$ as follows.

\begin{definition}\cite{Wenzl,Jones,Lick,KauLins}
Let $P_i\in TL_{n}(d)$ be defined inductively for $i\leq n$ by the following formulas:
\begin{eqnarray}
 P_1&=&1_{n}\\
 P_{i+1}&=& P_{i}-\frac{\Delta_{i-1}}{\Delta_{i}}P_{i}e_{i}P_{i}, ~provided~ i+1\leq n.
\end{eqnarray}
where the polynomial function $\Delta_i$ {of degree} $i$ in $d$ defined recursively by Eq. (\ref{E:Delta}).
\end{definition}

\begin{lemma}\cite{Wenzl,Jones,Lick,KauLins}
The elements $P_i\in TL_{n}(d)$ for $i\leq n$ enjoy the following properties.
\begin{enumerate}
 \item [(i)] $P_i^2=P_i.$
 \item [(ii)] $P_{i}e_k=e_kP_{i}$ for $k\geq i.$
\end{enumerate}
\end{lemma}

\begin{proposition}\cite{Wenzl,Jones,Lick,KauLins}
There is a unique non-zero element $P\in TL_{n}(d)$ such that
\begin{enumerate}
 \item [(i)] $P^2=P.$
 \item [(ii)] $Pe_{k}=0,$ for $k\leq n-1.$
\end{enumerate}
\end{proposition}

\begin{remark}
Note that in $TL_{n}(d)$ there are $n$ projectors satisfying $P_i^2=P_i$ for $i=1,2,...,n$. But only one projector $P_n$ satisfies $P^2=P$ and $Pe_{k}=0$ for $k\leq n-1.$
And Jones-Wenzl projector $P_2$ in $TL_2(d)$ can be used to construct a mapping between chromatic algebra and Temperley-Lieb algebra \cite{Krushkal}.
The trace radial $P_4$ (when $d=\frac{\sqrt{5}+1}{2}$, i.e., golden ratio) in $TL_4(d)$ can be used to prove a famous Tutte's formula (``golden identity") on chromatic polynomial. We refer to \cite{Krushkal} and references in it.
\end{remark}

Based on the definition and properties of Jones-Wenzl projectors, we shall generalize them to projectors of virtual Temperley-Lieb algebras.
First, we define two 1-variable functions $x_i$ and $z_i$ in $i$ as follows:
 $$x_i=\frac{1}{i+1},~z_i=\frac{i}{i+1}.$$
A 2-variable functions $y_i$ in $d,i$ is defined by
 $$y_i=-\frac{2i}{(i+1)(d+2i-2)}.$$
Let $\alpha_{i-1}=x_{i-1}d+y_{i-1}+z_{i-1}=\frac{(d+i-3)(d+2i-2)}{i(d+2i-4)}$.
Then
\begin{eqnarray}\label{key}
\frac{x_i+y_i\alpha_{i-1}}{-z_i}=\frac{d-2}{i(d+2i-4)}=x_{i-1}+y_{i-1}.
\end{eqnarray}
For example, $(x_0,y_0,z_0)=(1,0,0)$, $(x_1,y_1,z_1)=(\frac{1}{2},-\frac{1}{d},\frac{1}{2})$. $\alpha_0=d$ and $\alpha_1=\frac{d^2+d-2}{2d}$.

Next, we shall recursively construct the \textit{projector} $f_i$ of $VTL_{n}(d)$ as follows.

\begin{definition}\label{D:virtualpj}
Let $f_i\in VTL_{n}(d)$ be defined inductively for $i\leq n$ by the following formulas:
\begin{eqnarray}
 f_1&=&1_{n}\\
 f_{i}&=&x_{i-1}f_{i-1}+y_{i-1}f_{i-1}e_{i-1}f_{i-1}+z_{i-1}f_{i-1}v_{i-1}f_{i-1},~ provided~ i\leq n.
\end{eqnarray}
where the denominator of $y_{i-1}$ is non-zero, that is, $d\neq\{0,-2,-4,...,-2n+4\}$ (i.e., non-positive even number).
\end{definition}
For $f_{i-1}, f_{i}\in VTL_{n}(d)$, we can get the recursive relation between them as shown in Figure \ref{F:fn1virtualprojector}.
For example: In $VTL_2$, $f_2=\frac{1}{2}f_1-\frac{1}{d}f_1e_1f_1+\frac{1}{2}f_1v_1f_1
=\frac{1}{2}1_2-\frac{1}{d}e_1+\frac{1}{2}v_1$.
In particular, $f_i\in VTL_{n}(d)$ commutes with each $e_{i+k}, v_{i+k}$ for $k\geq 1,$ which will be used frequently later.

\begin{figure}[pb]
\centering
\includegraphics[width=5.5in]{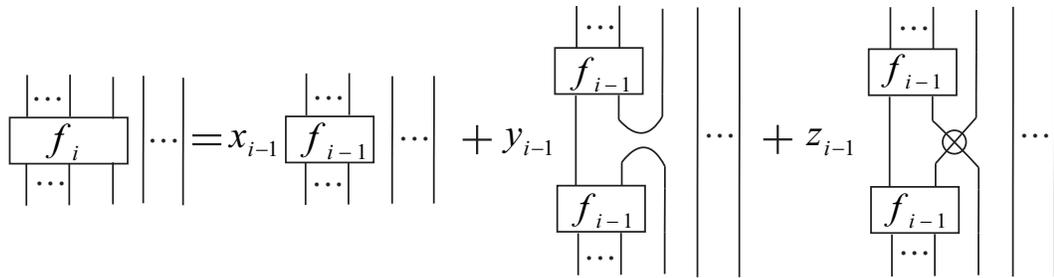}
\caption{The recursive relation between $f_{i}$ and $f_{i-1}$ in $VTL_{n}(d)$.}\label{F:fn1virtualprojector}
 \end{figure}

It is easy to get the following identities, which will be used frequently later.
\begin{itemize}
 \item [$(i)$] $z_{i}x_{i-1}=x_i$,
 \item [$(ii)$] $(y_{i}+z_i)y_{i-1}=y_{i}z_{i-1}$,
 \item [$(iii)$] $e_{i}v_{i\pm1}e_i=e_{i}$, $e_{i}v_{i\pm1}v_i=e_{i}e_{i\pm1}$, $v_i v_{i\pm1}e_{i}=e_{i\pm1}e _{i}$,
  $e_{i}e_{i\pm1}v_i=e_{i}v_{i\pm1}$ and $v_{i}e_{i\pm1}e_i=v_{i\pm1}e_{i}.$
\end{itemize}

%When the context is clear, $f_{i-1}1_n$ or $1_n f_{i-1}$ are abbreviated as $f_{i-1}.$
By the Definition \ref{D:virtualpj}, we can get the following Lemma,
which is one of main results in this article.

\begin{lemma}\label{new}
The projectors $f_i\in VTL_{n}(d)$ ($i\leq n$) enjoy the following properties.
\begin{itemize}
 \item [$(1_i)$] $f_i^2=f_i.$
 \item [$(2_i)$] $f_{i+1}f_i=f_if_{i+1}=f_{i+1}$, provided $i+1\leq n.$
 \item [$(3_i)$] $f_ie_k=e_kf_i=0$ for $k<i.$
 \item [$(4_i)$] $f_iv_k=v_kf_i=f_i$ for $k<i.$
 \item [$(5_i)$] $(e_{i}f_i)^2=\alpha_{i-1}e_if_i$ and $(f_{i}e_i)^2=\alpha_{i-1}f_ie_i.$
 \item [$(6_i)$] $f_iv_if_iv_if_i=x_{i-1}f_i+y_{i-1}f_ie_if_i+z_{i-1}f_iv_if_i.$
 \item [$(7_i)$] $e_if_iv_if_{i}=(x_{i-1}+y_{i-1})e_if_{i}$ and $f_{i}v_if_ie_i=(x_{i-1}+y_{i-1})f_{i}e_i.$
\end{itemize}
\end{lemma}
\begin{proof}
The Lemma is clearly true when $i=1$ for $1_i,2_i,5_i,6_i,7_i$ and $i=2$ for $3_i,4_i$.
Inductively suppose it is true for a given $i$,
and then suppose that $i+1\leq n$ and $d\neq\{0,-2,-4,...,-2n+4\}.$
This means that $f_1,f_2,...,f_{i+1}$ are all well defined.

$(1_{i+1})$
\begin{eqnarray*}
f_{i+1}^2&=&(x_{i}f_{i}+y_{i}f_{i}e_{i}f_{i}+z_{i}f_{i}v_{i}f_{i})^2\\
&=&x_{i}^{2}f_{i}+y_{i}^{2}f_{i}e_{i}f_{i}e_{i}f_{i}+z_{i}^{2}f_{i}v_{i}f_{i}v_{i}f_{i}+2x_{i}y_{i}f_{i}e_{i}f_{i}+2x_{i}z_{i}f_{i}v_{i}f_{i}\\
& &+y_{i}z_{i}f_{i}e_{i}f_{i}v_{i}f_{i}+z_{i}y_{i}f_{i}v_{i}f_{i}e_{i}f_{i} \ \ \ (by \  \ 1_i)\\
&=&x_{i}^{2}f_{i}+y_{i}^{2}\alpha_{i-1}f_{i}e_{i}f_{i}+z_{i}^{2}(x_{i-1}f_i+y_{i-1}f_ie_if_i+z_{i-1}f_iv_if_i)\\
& &+2x_{i}y_{i}f_{i}e_{i}f_{i}+2x_{i}z_{i}f_{i}v_{i}f_{i}+2y_{i}z_{i}(x_{i-1}+y_{i-1})f_ie_if_i \ \ \ (by \ \ 5_i, \ 6_i, \  7_i)\\
&=&(x_i^2+x_{i-1}z_{i}^{2})f_i+(y_{i}^{2}\alpha_{i-1}+z_i^2y_{i-1}+2x_{i}y_{i}+2y_{i}z_{i}(x_{i-1}+y_{i-1}))f_ie_if_i\\ & & +(z_{i-1}z_i^2+2x_iz_i)f_{i}v_{i}f_{i}
\end{eqnarray*}
\begin{eqnarray*}
&\because&x_i^2+x_{i-1}z_{i}^{2}=\frac{1}{(i+1)^2}+\frac{1}{i}\frac{i^2}{(i+1)^2}=\frac{1}{i+1}=x_{i}\\
& &y_{i}^{2}\alpha_{i-1}+z_i^2y_{i-1}+2x_{i}y_{i}+2y_{i}z_{i}(x_{i-1}+y_{i-1})\\
& &=z_i^2y_{i-1}-y_{i}^2\alpha_{i-1}\ \ \ (by \ \ (\ref{key}))\\
& &=\frac{i^2}{(i+1)^2}\frac{-2(i-1)}{i(d+2i-4)}
-(\frac{-2i}{(i+1)(d+2i-2)})^2\frac{(d+i-3)(d+2i-2)}{i(d+2i-4)}\\
& &=-\frac{2i}{(i+1)(d+2i-2)}\\
& &=y_i\\
& &z_{i-1}z_i^2+2x_iz_i=z_i(z_{i-1}z_i+2x_i)=z_i\\
&\therefore&f_{i+1}^2=x_{i}f_{i}+y_{i}f_{i}e_{i}f_{i}+z_{i}f_{i}v_{i}f_{i}=f_{i+1}
\end{eqnarray*}
$(2_{i+1})$\ \ Provided $i+2\leq n,$ then
\begin{eqnarray*}
&f_{i+2}f_{i+1}&=x_{i+1}f_{i+1}^2+y_{i+1}f_{i+1}e_{i+1}f_{i+1}^2+z_{i+1}f_{i+1}v_{i+1}f_{i+1}^2=f_{i+2}\\
&f_{i+1}f_{i+2}&=x_{i+1}f_{i+1}^2+y_{i+1}f_{i+1}^2e_{i+1}f_{i+1}+z_{i+1}f_{i+1}^2v_{i+1}f_{i+1}=f_{i+2}
\end{eqnarray*}
$(3_{i+1})$\ \ If $k<i,$ then
\begin{eqnarray*}
e_k f_{i+1}&=&e_k(x_{i}f_{i}+y_{i}f_{i}e_{i}f_{i}+z_{i}f_{i}v_{i}f_{i})=0 \ \  (by \ \  3_i),\\
f_{i+1}e_k&= &(x_{i}f_{i}+y_{i}f_{i}e_{i}f_{i}+z_{i}f_{i}v_{i}f_{i})e_k=0 \ \  (by \ \  3_i).\\
e_i f_{i+1}&=&e_i(x_{i}f_{i}+y_{i}f_{i}e_{i}f_{i}+z_{i}f_{i}v_{i}f_{i})\\
 &=&x_{i}e_i f_{i}+y_{i}\alpha_{i-1}e_{i}f_{i}+z_{i}(x_{i-1}+y_{i-1})e_{i}f_i \ \  (by  \ \ 5_i, \ \ 7_i, (\ref{key}))\\
  &=&0\\
f_{i+1}e_i&=&(x_{i}f_{i}+y_{i}f_{i}e_{i}f_{i}+z_{i}f_{i}v_{i}f_{i})e_i\\
  &=&x_{i}f_{i}e_i+y_{i}\alpha_{i-1}f_{i}e_{i}+z_{i}(x_{i-1}+y_{i-1})f_i e_{i} \ \  (by  \ \ 5_i,  \ \ 7_i, (\ref{key}))\\
  &=&0
\end{eqnarray*}

$(4_{i+1})$\ \ It is easy to prove the case $k<i$. In the following we consider the case $k=i$.
\begin{eqnarray*}
f_{i+1}v_i&=&x_{i}f_{i}v_i+y_{i}f_{i}e_{i}f_{i}v_i+z_{i}f_{i}v_{i}f_{i}v_i\\
&=&x_{i}f_{i}v_i+y_{i}f_{i}e_{i}(x_{i-1}f_{i-1}+y_{i-1}f_{i-1}e_{i-1}f_{i-1}+z_{i-1}f_{i-1}v_{i-1}f_{i-1})v_i\\
& &+z_{i}f_{i}v_{i}(x_{i-1}f_{i-1}+y_{i-1}f_{i-1}e_{i-1}f_{i-1}+z_{i-1}f_{i-1}v_{i-1}f_{i-1})v_i\\
&=&x_{i}f_{i}v_i+y_{i}x_{i-1}f_{i}e_{i}f_{i-1}v_i+y_{i}y_{i-1}f_{i}e_{i}f_{i-1}e_{i-1}f_{i-1}v_i\\
& &+y_{i}z_{i-1}f_{i}e_{i}f_{i-1}v_{i-1}f_{i-1}v_i+z_{i}x_{i-1}f_{i}v_{i}f_{i-1}v_i\\
& &+z_{i}y_{i-1}f_{i}v_{i}f_{i-1}e_{i-1}f_{i-1}v_i+z_{i}z_{i-1}f_{i}v_{i}f_{i-1}v_{i-1}f_{i-1}v_i\\
&=&x_{i}f_{i}v_i+y_{i}x_{i-1}f_{i}f_{i-1}e_{i}v_i+y_{i}y_{i-1}f_{i}f_{i-1}e_{i}e_{i-1}v_if_{i-1}\\
& &+y_{i}z_{i-1}f_{i}f_{i-1}e_{i}v_{i-1}v_if_{i-1}+z_{i}x_{i-1}f_{i}f_{i-1}v_{i}v_i\\
& &+z_{i}y_{i-1}f_{i}f_{i-1}v_{i}e_{i-1}v_if_{i-1}\\
& &+z_{i}z_{i-1}f_{i}f_{i-1}v_{i}v_{i-1}v_i f_{i-1}\ \ \  (using ~ U_if_{i-1}=f_{i-1}U_i,~U_i=e_i~or~v_i) \\
&=&x_{i}f_{i}v_i+y_{i}x_{i-1}f_{i}e_{i}+y_{i}y_{i-1}f_{i}e_{i}e_{i-1}v_if_{i-1}
+y_{i}z_{i-1}f_{i}e_{i}v_{i-1}v_if_{i-1}\\
& &+z_{i}x_{i-1}f_{i}+z_{i}y_{i-1}f_{i}v_{i}e_{i-1}v_if_{i-1}
+z_{i}z_{i-1}f_{i}v_{i}v_{i-1}v_i f_{i-1} \ \ \ (using \ 2_{i-1})\\
&=&x_{i}f_{i}v_i+y_{i}x_{i-1}f_{i}e_{i}+y_{i}y_{i-1}f_{i}e_{i}v_{i-1}f_{i-1}\\
& &+y_{i}z_{i-1}f_{i}e_{i}e_{i-1}f_{i-1}+z_{i}x_{i-1}f_{i}\\
& &+z_{i}y_{i-1}f_{i}v_{i-1}e_{i}v_{i-1}f_{i-1}
+z_{i}z_{i-1}f_{i}v_{i-1}v_{i}v_{i-1}f_{i-1} \ \ \ (using \  (iii))\\
&=&x_{i}f_{i}v_i+y_{i}x_{i-1}f_{i}e_{i}+y_{i}y_{i-1}f_{i}e_{i}v_{i-1}f_{i-1}
+y_{i}z_{i-1}f_{i}e_{i}e_{i-1}f_{i-1}\\
& &+z_{i}x_{i-1}f_{i}+z_{i}y_{i-1}f_{i}e_{i}v_{i-1}f_{i-1}
+z_{i}z_{i-1}f_{i}v_{i}v_{i-1}f_{i-1} \ \ \ (using \ 4_i)\\
&=&x_{i}f_{i}v_i+y_{i}x_{i-1}f_{i}e_{i}+(y_{i}y_{i-1}+z_{i}y_{i-1})f_{i}e_{i}v_{i-1}f_{i-1}\\
& &+y_{i}z_{i-1}f_{i}e_{i}e_{i-1}f_{i-1}+z_{i}x_{i-1}f_{i}
+z_{i}z_{i-1}f_{i}v_{i}v_{i-1}f_{i-1}\\
&=&x_{i}f_{i}v_i+y_{i}x_{i-1}f_{i}e_{i}+y_{i}z_{i-1}f_{i}e_{i}v_{i-1}f_{i-1}\\
& &+y_{i}z_{i-1}f_{i}e_{i}e_{i-1}f_{i-1}+x_{i}f_{i}
+z_{i}z_{i-1}f_{i}v_{i}v_{i-1}f_{i-1} \ \ \ (using \ (i)(ii))
\end{eqnarray*}
\begin{eqnarray*}
f_{i+1}&=&x_{i}f_{i}+y_{i}f_{i}e_{i}f_{i}+z_{i}f_{i}v_{i}f_{i}\\
&=&x_{i}f_{i}+y_{i}f_{i}e_{i}(x_{i-1}f_{i-1}+y_{i-1}f_{i-1}e_{i-1}f_{i-1}+z_{i-1}f_{i-1}v_{i-1}f_{i-1})\\
& &+z_{i}f_{i}v_{i}(x_{i-1}f_{i-1} +y_{i-1}f_{i-1}e_{i-1}f_{i-1}+z_{i-1}f_{i-1}v_{i-1}f_{i-1})\\
&=&x_{i}f_{i}+y_{i}x_{i-1}f_{i}e_{i}f_{i-1}+y_{i}y_{i-1}f_{i}e_{i}f_{i-1}e_{i-1}f_{i-1}\\
& &+y_{i}z_{i-1}f_{i}e_{i}f_{i-1}v_{i-1}f_{i-1}+z_{i}x_{i-1}f_{i}v_{i}f_{i-1}\\
& &+z_{i}y_{i-1}f_{i}v_{i}f_{i-1}e_{i-1}f_{i-1}
+z_{i}z_{i-1}f_{i}v_{i}f_{i-1}v_{i-1}f_{i-1}
\end{eqnarray*}
\begin{eqnarray*}
&=&x_{i}f_{i}+y_{i}x_{i-1}f_{i}f_{i-1}e_{i}+y_{i}y_{i-1}f_{i}f_{i-1}e_{i}e_{i-1}f_{i-1}\\
& &+y_{i}z_{i-1}f_{i}f_{i-1}e_{i}v_{i-1}f_{i-1}+z_{i}x_{i-1}f_{i}f_{i-1}v_{i}+z_{i}y_{i-1}f_{i}f_{i-1}v_{i}e_{i-1}f_{i-1}\\
& &+z_{i}z_{i-1}f_{i}f_{i-1}v_{i}v_{i-1}f_{i-1}\ \ \  (using \ U_if_{i-1}=f_{i-1}U_i, U_i=e_i~or~v_i) \\
&=&x_{i}f_{i}+y_{i}x_{i-1}f_{i}e_{i}+y_{i}y_{i-1}f_{i}e_{i}e_{i-1}f_{i-1}+y_{i}z_{i-1}f_{i}e_{i}v_{i-1}f_{i-1}\\
& &+z_{i}x_{i-1}f_{i}v_{i}+z_{i}y_{i-1}f_{i}v_{i}e_{i-1}f_{i-1}+z_{i}z_{i-1}f_{i}v_{i}v_{i-1}f_{i-1} \ \ \ (by \ 2_{i-1})
\end{eqnarray*}
\begin{eqnarray*}
& &f_{i+1}v_i-f_{i+1}\\
&=&x_{i}f_{i}v_i+y_{i}x_{i-1}f_{i}e_{i}+y_{i}z_{i-1}f_{i}e_{i}v_{i-1}f_{i-1}\\
& &+y_{i}z_{i-1}f_{i}e_{i}e_{i-1}f_{i-1}+x_{i}f_{i}+z_{i}z_{i-1}f_{i}v_{i}v_{i-1}f_{i-1}\\
& &-[x_{i}f_{i}+y_{i}x_{i-1}f_{i}e_{i}+y_{i}y_{i-1}f_{i}e_{i}e_{i-1}f_{i-1}+y_{i}z_{i-1}f_{i}e_{i}v_{i-1}f_{i-1}\\
& &+z_{i}x_{i-1}f_{i}v_{i}+z_{i}y_{i-1}f_{i}v_{i}e_{i-1}f_{i-1}+z_{i}z_{i-1}f_{i}v_{i}v_{i-1}f_{i-1}]\\
&=&y_{i}(z_{i-1}-y_{i-1})f_{i}e_{i}e_{i-1}f_{i-1}-z_{i}y_{i-1}f_{i}v_{i}e_{i-1}f_{i-1}\\
&=&y_{i-1}z_{i}f_{i}e_{i}e_{i-1}f_{i-1}-z_{i}y_{i-1}f_{i}v_{i}e_{i-1}f_{i-1} \ \ \ (using \ (ii))\\
&=&y_{i-1}z_{i}[f_{i}e_{i}e_{i-1}f_{i-1}-f_{i}v_{i}e_{i-1}f_{i-1}]
\end{eqnarray*}
\begin{eqnarray*}
& &f_{i}e_{i}e_{i-1}f_{i-1}v_i-f_{i}v_{i}e_{i-1}f_{i-1}v_i\\
&=&f_{i}e_{i}e_{i-1}v_i f_{i-1}-f_{i}v_{i}e_{i-1}v_i f_{i-1}\\
&=&f_{i}e_{i}v_{i-1}f_{i-1}-f_{i}v_{i-1}e_{i}v_{i-1}f_{i-1}\\
&=&f_{i}e_{i}v_{i-1}f_{i-1}-f_{i}e_{i}v_{i-1}f_{i-1} \ \ \ (using \ 4_i) \\
&=&0
\end{eqnarray*}
Hence $f_{i}e_{i}e_{i-1}f_{i-1}v_i=f_{i}v_{i}e_{i-1}f_{i-1}v_i$

Then $f_{i}e_{i}e_{i-1}f_{i-1}v_iv_i=f_{i}v_{i}e_{i-1}f_{i-1}v_iv_i$, i.e., $f_{i}e_{i}e_{i-1}f_{i-1}=f_{i}v_{i}e_{i-1}f_{i-1}.$

Therefore $f_{i+1}v_i-f_{i+1}=0$, i.e., $f_{i+1}v_i=f_{i+1}.$

With the same reason, we can get $v_if_{i+1}=f_{i+1}.$

$(5_{i+1})$ \ \ Since $f_{i}e_{i+1}=e_{i+1}f_{i}$, $e_{i+1}f_{i}e_{i+1}=d e_{i+1}f_{i}$, $e_{i+1}e_{i}e_{i+1}=e_{i+1}$
and $e_{i+1}v_{i}e_{i+1}=e_{i+1}$, then
\begin{eqnarray*}
(e_{i+1}f_{i+1})^2 &=&e_{i+1}(x_{i}f_{i}+y_{i}f_{i}e_{i}f_{i}+z_{i}f_{i}v_{i}f_{i})e_{i+1}f_{i+1}\\
&=&(d x_{i}e_{i+1}f_{i}+y_{i}f_{i}e_{i+1}f_{i}+z_{i}f_{i}e_{i+1}f_{i})f_{i+1}\\
&=&d x_{i}e_{i+1}f_{i+1}+y_{i}e_{i+1}f_{i+1}+z_{i}e_{i+1}f_{i+1} \ \ (from \ \ 2_i).\\
&=&(d x_{i}+y_{i}+z_{i})e_{i+1}f_{i+1}\\
&=&\alpha_{i}e_{i+1}f_{i+1}
\end{eqnarray*}

With the same reason, we can get
$(f_{i+1}e_{i+1})^2=\alpha_{i}f_{i+1}e_{i+1}.$

$(6_{i+1})$ \ \ Since $f_{i}v_{i+1}=v_{i+1}f_{i}$, $v_{i+1}^2=1_{n}$
,$v_{i+1}e_{i}v_{i+1}=v_{i}e_{i+1}v_{i}$ and $v_{i+1}v_{i}v_{i+1}$ $=v_{i}v_{i+1}v_{i}$, then
\begin{eqnarray*}
& &f_{i+1}v_{i+1}f_{i+1}v_{i+1}f_{i+1}\\
&=&f_{i+1}v_{i+1}(x_{i}f_{i}+y_{i}f_{i}e_{i}f_{i}+z_{i}f_{i}v_{i}f_{i})v_{i+1}f_{i+1}\\
&=&x_{i}f_{i+1}f_{i}v_{i+1}v_{i+1}f_{i+1}+y_{i}f_{i+1}f_{i}v_{i+1}e_{i}v_{i+1}f_{i}f_{i+1}\\
& &+z_{i}y_{i}f_{i+1}f_{i}v_{i+1}v_{i}v_{i+1}f_{i}f_{i+1}\\
&=&x_{i}f_{i+1}+y_{i}f_{i+1}v_{i}e_{i+1}v_{i}f_{i+1}
+z_{i}y_{i}f_{i+1}v_{i}v_{i+1}v_{i}f_{i+1} \ \ \ (using \ 2_i,\ 1_{i+1}) \\
&=&x_{i}f_{i+1}+y_{i}f_{i+1}e_{i+1}f_{i+1}+z_{i}f_{i+1}v_{i+1}f_{i+1} \ \ (from \ \ 4_i)
\end{eqnarray*}

$(7_{i+1})$
\begin{eqnarray*}
& &f_{i+1}v_{i+1}f_{i+1}e_{i+1}\\
&=&f_{i+1}v_{i+1}(x_{i}f_{i}+y_{i}f_{i}e_{i}f_{i}+z_{i}f_{i}v_{i}f_{i})e_{i+1}\\
&=&x_{i}f_{i+1}f_{i}v_{i+1}e_{i+1}+y_{i}f_{i+1}f_{i}v_{i+1}e_{i}e_{i+1}f_{i}+z_{i}f_{i+1}f_{i}v_{i+1}v_{i}e_{i+1}f_{i}\\
&=&x_{i}f_{i+1}v_{i+1}e_{i+1}+y_{i}f_{i+1}v_{i+1}e_{i}e_{i+1}f_{i}+z_{i}f_{i+1}v_{i+1}v_{i}e_{i+1}f_{i}\ \ (using \ 2_i)\\
&=&x_{i}f_{i+1}e_{i+1}+y_{i}f_{i+1}v_{i}e_{i+1}f_{i}+z_{i}f_{i+1}e_{i}e_{i+1}f_{i}\ \ (using \ (iii))\\
&=&x_{i}f_{i+1}e_{i+1}+y_{i}f_{i+1}e_{i+1}f_{i}\ \ (using \ 4_{i+1},\ 3_{i+1})\\
&=&(x_{i}+y_i)f_{i+1}e_{i+1}
\end{eqnarray*}
With the same reason, we can get
$e_{i+1}f_{i+1}v_{i+1}f_{i+1}=(x_{i}+y_{i})e_{i+1}f_{i+1}.$
\end{proof}

\begin{remark}
Note that in $VTL_{n}(d)$ there are $n$ projectors satisfying $f_i^2=f_i$ for $i=1,2,...,n$.
And there exists one projector $f_n$ satisfies $f_n^2=f_n$, $f_ne_{k}=0,$ and $f_nv_{k}=f_n$ for $k\leq n-1.$
\end{remark}

\begin{proposition}\label{unique}
There is a unique non-zero element $f\in VTL_{n}(d)$ such that
\begin{enumerate}
\item [(i)] $f^2=f$,
\item [(ii)] $fe_i=e_if=0$, $i=1,2,...,n-1,$
\item [(iii)] $fv_i=v_if=f$, $i=1,2,...,n-1.$
\end{enumerate}
\end{proposition}

\begin{proof}
Lemma \ref{new}($1_i$)($3_i$)($4_i$) asserts the existence of the element. Now we prove the uniqueness. Recall that there are $(2n-1)!!$ elements in a base.
It is easy to note that there are $n!$ elements which do not contain any $e_i$, i.e., they are the products of the elements $\{1_{n},v_1,v_2,...,v_{n-1}\}$.
We denote them by $\{\epsilon_1,\epsilon_2$, $...,\epsilon_{n!}\}$ where $\epsilon_1=1_{n}$.
And we denote the other elements in the base by $\{\epsilon_{n!+1},$ $\epsilon_{n!+2},...,\epsilon_{(2n-1)!!}\}$.

For any $f\in VTL_{n}(d)$, if it satisfies (ii) and (iii), then $f\epsilon_i=f$ for $i=1,2,...,n!$ by (iii) and $f\epsilon_i=0$ for $i=n!+1,n!+2,...,(2n-1)!!$ by (ii) and (iii).
Let $f=\sum_{i=1}^{(2n-1)!!}x_i\epsilon_i$. Note that for each $k\in\{1,2,...,n!\}$, there exists unique $j\in\{1,2,...,n!\}$ such that $\epsilon_j=\epsilon_k^{-1}$ and $\epsilon_k*\epsilon_k^{-1}=\epsilon_1=1_{n}$.
Then $\epsilon_jf=\epsilon_j\sum_{i=1}^{(2n-1)!!}x_i\epsilon_i=x_k\epsilon_1+\sum_{i\neq k}x_i\epsilon_j\epsilon_i=f=x_1\epsilon_1+\sum_{i\neq 1}x_i\epsilon_i$.
Note that $\epsilon_j\epsilon_i\neq\epsilon_j\epsilon_{l}$ if $i\neq l$ and there does not exist $\epsilon_j\epsilon_i=\epsilon_1$ for some $i\neq k$.
Thus $x_k=x_1$.

Moreover, $f=f^2=f\sum_{i=1}^{(2n-1)!!}x_i\epsilon_i=\sum_{i=1}^{n!}x_if$, then $1=\sum_{i=1}^{n!}x_i$ and $x_k=\frac{1}{n!}$ for $k\in\{1,2,...,n!\}$. It means that $f=\frac{1}{n!}\sum_{i=1}^{n!}\epsilon_i+\sum_{i=n!+1}^{(2n-1)!!}x_i\epsilon_i$.

Suppose that $g=\frac{1}{n!}\sum_{i=1}^{n!}\epsilon_i+\sum_{i=n!+1}^{(2n-1)!!}y_i\epsilon_i$ also satisfies $(i)(ii)(iii)$. Then $g\epsilon_i=\epsilon_ig=g$ for $i\in\{1,2,...,n!\}$ and $g\epsilon_i=\epsilon_ig=0$ for $i\in \{n!+1,n!+2,...,(2n-1)!!\}.$
Then $(f-\frac{1}{n!}\sum_{i=1}^{n!}\epsilon_i)g=0$ and $f(g-\frac{1}{n!}\sum_{i=1}^{n!}\epsilon_i)=0$. Hence $g=fg=f$.
\end{proof}
\skip 0.5cm

\section{Simplifications of the recurrence formula}
Let $f_n\in VTL_{n}(d)$ and $f_{n+1}\in VTL_{n+1}(d)$ be the projectors which satisfy the three conditions in Proposition \ref{unique}.
We call $f_n$ the $n$-th projector of virtual Temperley-Lieb algebra.
The $(n+1)$-th projector is constructed from the $n$-th projector as shown in Figure \ref{F:fn1virtualprojector}.
In $f_{n+1}=x_{n}f_{n}+y_{n}f_{n}e_{n}f_{n}+z_{n}f_{n}v_{n}f_{n}$, $x_nf_n$ means $x_nf_n1_{n+1}$.
By expanding this appropriately, we will see that many of the terms do not contribute.

\begin{figure}
\centering
\includegraphics{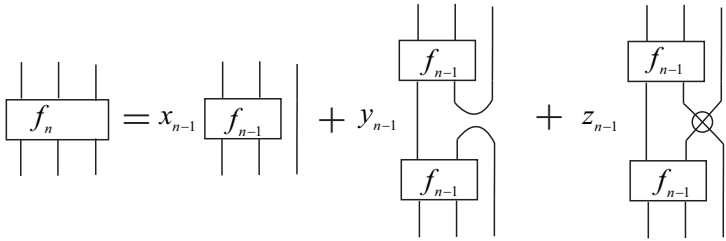}
\caption{The recurrence relation between projectors.}
\label{F:fnfn1}
\end{figure}

We define $\mathcal{K}_{n}\subset VTL_{n}(d)$ as the linear span of diagrams $D$s which belong to the set $A_{n}$, where

\begin{eqnarray}
A_{n}&=&\{1_{n},U_{n-1}U_{n-2}\cdots U_{i+1}U_{i}|1\leq i\leq n-1, U_{j}=e_j ~\text{or}~ v_j\}.
\label{E:An}
\end{eqnarray}

This is illustrated for $n=3$ in Figure \ref{F:K3}.
It is easy to obtain that $dim(\mathcal{K}_2)=3$, $dim(\mathcal{K}_3)=7$ and $dim(\mathcal{K}_4)=15$.
By simple calculating, we see $\mathcal{K}_n$ has dimension $2^{n}-1$.

\begin{figure}[!htbp]
\centering
\includegraphics[width=4.0in]{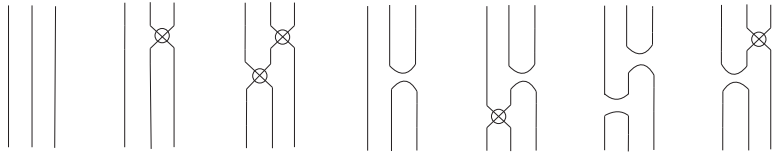}
  \caption{The diagrams spanning $\mathcal{K}_3$ of the set $A_3$.}\label{F:K3}
\end{figure}

We shall construct $f_n^{\mathcal{K}}\in \mathcal{K}_n$ by using the following recursion:
\begin{eqnarray}
f_2^{\mathcal{K}}&=&x_{1}1_{2}+y_{1}e_{1}+z_{1}v_{1}\\
f_{n}^{\mathcal{K}}&=&x_{n-1}1_{n}+y_{n-1}e_{n-1}f_{n-1}^\mathcal{K}+z_{n-1}v_{n-1}f_{n-1}^\mathcal{K}
\end{eqnarray}

Next we can simplify the recurrence formula for projectors as follows.

\begin{lemma}\label{L:fmK}
Let $f_{n}$ be the projector of $VTL_{n}(d)$ which satisfies the three conditions in Proposition \ref{unique}.
Then we claim that
\begin{equation}\label{E:simplifiedformula}
  f_{n}=f_{n-1}f_{n}^{\mathcal{K}}.
\end{equation}
\end{lemma}
\begin{proof}
We induct on $n$.
Note that $f_{2}^{K}=f_2$. The Lemma is clearly true when $n=2$, that is, $f_{2}=f_{1}f_{2}^{\mathcal{K}}.$ Inductively suppose it is true for a given $n$, and then for $n+1$, we have
\begin{eqnarray*}
f_{n+1}&=&x_{n}f_{n}+y_{n}f_{n}e_{n}f_{n}+z_{n}f_{n}v_{n}f_{n} \\
&=&x_{n}f_{n}+y_{n}f_{n}e_{n}f_{n-1}f_{n}^{\mathcal{K}}+z_{n}f_{n}v_{n}f_{n-1}f_{n}^{\mathcal{K}}\\
&=&x_{n}f_{n}+y_{n}f_{n}f_{n-1}e_{n}f_{n}^{\mathcal{K}}+z_{n}f_{n}f_{n-1}v_{n}f_{n}^{\mathcal{K}}\\
&=&x_{n}f_{n}+y_{n}f_{n}e_{n}f_{n}^{\mathcal{K}}+z_{n}f_{n}v_{n}f_{n}^{\mathcal{K}}~~(using~ f_nf_{n-1}=f_{n})\\
&=&f_{n}(x_{n}1_{n+1}+y_{n}e_{n}f_{n}^{\mathcal{K}}+z_{n}v_{n}f_{n}^{\mathcal{K}})\\
&=&f_{n}f_{n+1}^{\mathcal{K}}
\end{eqnarray*}
where we use $e_nf_{n-1}=f_{n-1}e_{n}$, $v_nf_{n-1}=f_{n-1}v_{n}$ in the second equation, $f_nf_{n-1}=f_n$ in the third equation.
\end{proof}

\begin{remark}
Recall that $f_{n}=f_{n-1}[x_{n-1}1_{n}+y_{n-1}e_{n-1}f_{n-1}+z_{n-1}v_{n-1}f_{n-1}]$.
But now we have that $f_{n}=f_{n-1}[x_{n-1}1_{n}+y_{n-1}e_{n-1}f_{n-1}^\mathcal{K}+z_{n-1}v_{n-1}f_{n-1}^\mathcal{K}]$, which is more simpler.
\end{remark}

\section{The coefficient of projectors}
It is obvious that we can construct a natural mapping from $VTL_{n}(d)$ to symmetric group $S_{2n}$.
Then every element of $VTL_{n}(d)$ can be represented by a permutation which is the product of $n$ non-intersecting transpositions.
Note that each connect element of $VTL_{n}(d)$ consists of a number of straight strands and some turn-back.
A $\textit{through strand}$ \cite{Mor} is an arc joining with a point on the top edge of a diagram of an element of $VTL_{n}(d)$ to another point on the bottom edge.
A $\textit{vertical strand}$ is an arc joining with a point on the top edge of a diagram of an element of $VTL_{n}(d)$ to the corresponding point on the bottom edge.
A \textit{cup} joins a point on the top edge with another point on top edge, and similarly a \textit{cap} joins a point on the bottom edge with another point on bottom edge.
A cap or cup is called a \textit{cap on site $i$} if it connects continuous two point on bottom or top edge.
This terminology is illustrated in Figure \ref{F:terminology}.

We shall call an element with $k$ through strands as a \textit{$k$-element} of $VTL_{n}(d)$.
Denote the set of $k$-elements by $\mathcal{E}_n^k$.
In Proposition \ref{unique}, we take the set of $n$-elements as $\{\epsilon_1,\epsilon_2$, $...,\epsilon_{n!}\}$ where $\epsilon_1=1_{n}$. That is, $\mathcal{E}_n^n=\{\epsilon_1,\epsilon_2$, $...,\epsilon_{n!}\}$.
Note that $(\mathcal{E}_n^n,\cdot)$ forms a multiplicative group under the product in $VTL_{n}(d)$.
Let $[k]_n$ be the sum over all $k$-elements of $VTL_{n}(d)$.
Note that $k=n,n-2,n-4,...,n-2[\frac{n}{2}]$ and $n-k=2l$ ($0\leq l\leq [\frac{n}{2}]$).
Note that there has at most one cap and one cup in any element of ${A}_{n}$ (see Eq. (\ref{E:An})), that is, each element of ${A}_{n}$ is either $(n-2)$-element or $n$-element.

\begin{figure}[!htbp]
\centering
\includegraphics{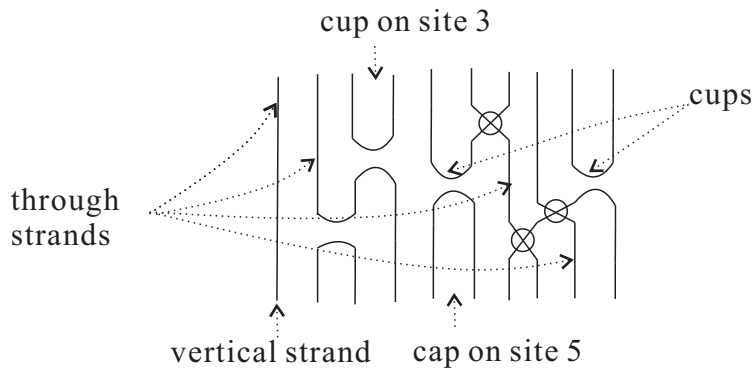}
\caption{The terminology.}\label{F:terminology}
\end{figure}

\begin{lemma}\label{kts}
For each $k\in \{n,n-2,n-4,...,n-2[\frac{n}{2}]\}$, the product between any $k$-element of $VTL_{n}(d)$ with $\alpha\in \mathcal{E}_n^n$ is still a $k$-element of $VTL_{n}(d)$.
\end{lemma}
\begin{proof}
It is because of $\alpha$ permutes a cap (resp. cup) and a through strand to a cap (resp. cup) strand and a through strand for any $k$-element multiplying by $\alpha$ on the right (resp. left), respectively.
\end{proof}

\begin{lemma}\label{kel}
For each $k\in \{n,n-2,n-4,...,n-2[\frac{n}{2}]\}$, given any two $k$-elements $x,y$ of $VTL_{n}(d)$, then there exists elements $\alpha,\beta\in \mathcal{E}_n^n$ such that $y=\alpha x\beta$.
\end{lemma}
\begin{proof}
The elements $\alpha$ and $\beta$ rearrange the order of points on top edge and bottom edge, respectively, meanwhile do not change the number of through strands by Lemma \ref{kts}.
\end{proof}

\begin{proposition}\label{coef}
Let $f_n=\sum_{i=1}^{(2n-1)!!}x_i\epsilon_i$ be the projector of $VTL_{n}(d)$.
For each $k\in \{n,n-2,n-4,...,n-2[\frac{n}{2}]\}$, given any two $k$-elements $x,y$ of $VTL_{n}(d)$, then the coefficients of $x,y$ are equivalent.
\end{proposition}
\begin{proof}
By Lemma \ref{kel}, there exists elements $\alpha, \beta\in \mathcal{E}_n^n$ such that $y=\alpha x\beta$. Then $\alpha f_n\beta=f_n$ by Proposition \ref{unique} $(iii)$. It is obvious that there does not exist another $k$-element $z$ of $VTL_{n}(d)$ such that $\alpha z \beta=y$, if does, then we can get $x=z$ since there exists inverse element of $\alpha$ (resp. $\beta$). Therefore, the coefficients of $x,y$ are equivalent.
\end{proof}

We denote $f_n=\sum_{l=0}^{[\frac{n}{2}]}\coeff{n}{}([n-2l]_n)[n-2l]_n$.
Then $\coeff{n}{}([n]_n)=\frac{1}{n!}$ by Proposition \ref{unique}.

\begin{figure}[!htbp]
\centering
\includegraphics[width=1.5in]{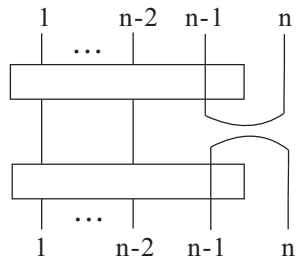}
\caption{The term $e_{n-1}$ in $f_{n}$. }\label{F:projectorem}
\end{figure}

\begin{proposition}\label{P:coefm-1}
The coefficient of $e_{n-1}$ in $f_n$ is equal to $-\frac{2}{n!(d+2n-4)}$, i.e., $\coeff{n}{}([n-2]_n)=-\frac{2}{n!(d+2n-4)}.$
\end{proposition}
\begin{proof}
Based on the recursive formula of $f_n$ in Figure \ref{F:fnfn1}, $e_{n-1}$ only comes from
$f_{n-1}e_{n-1}f_{n-1}$, where there exists two elements $\alpha$ and $\beta$ of $VTL_{n-1}(d)$ such that $\alpha\cdot1_n$ and $\beta\cdot1_n$ belong to $\mathcal{E}_{n}^{n}$ and the last one strand in both $\alpha$ and $\beta$ are vertical strand, $\alpha e_{n-1}\beta=e_{n-1}$ and $\alpha\beta=1_{n-1}$. See Figure \ref{F:projectorem}.
Since there are $(n-2)!$ choices for such $\alpha$ and $\beta$ and the coefficients of $\alpha$ and $\beta$ both are $\frac{1}{(n-1)!}$ in $f_{n-1}$,
then we can obtain that the coefficient of $e_{n-1}$ in $f_n$ is equal to $(n-2)!\frac{1}{(n-1)!^2}y_{n-1}$, i.e., $\coeff{n}{}([n-2]_n)=-\frac{2}{n!(d+2n-4)}.$
\end{proof}

\section{An explicit formula for the projectors}
The aim of this section is to present a method for calculating the coefficients for each diagram appearing in the projector $f_n$. The starting point will be the simplified recurrence formula   Eq.(\ref{E:simplifiedformula}), allowing us to calculate $f_n$ in terms of $f_{n-1}$ and given $f_{n}^{\mathcal{K}}$.

We consider a special type of $k$-element $e_{k+1}e_{k+3}\cdots e_{n-1}$ ($k<n$) in $VTL_n(d)$, which is the diagram with first $k$ vertical strands and $\frac{n-k}{2}$ (recall that $n-k=2l$) cap-cup pairs.
We call as \textit{canonical $k$-element} of $VTL_n(d)$ and denote by $\mathcal{CE}_{n}^{k}$, that is, $\mathcal{CE}_{n}^{k}=e_{k+1}e_{k+3}\cdots e_{n-1}$.
Note that there are cups and caps on sites $k+1$, $k+3$,...,$n-1$ in $\mathcal{CE}_{n}^{k}$.
This is illustrated for $n=5,6$ in Figure \ref{F:canonicalelement}.
In this section, we first attempt to get the coefficient of $\mathcal{CE}_{n}^{k}$ ($=\coeff{n}{}([k]_n)$) such that we can obtain an explicit formula for projector $f_n$.

\begin{figure}[!htbp]
\centering
\includegraphics[width=4.5in]{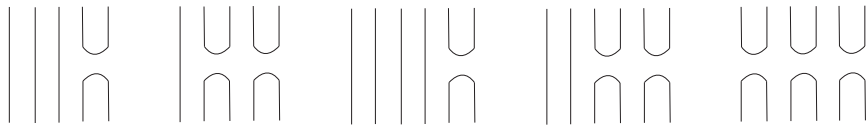}
  \caption{The canonical elements of $VTL_{5}$ and $VTL_{6}$.}\label{F:canonicalelement}
\end{figure}

To get the canonical $k$-element $\mathcal{CE}_{n+1}^{k}$ of $VTL_{n+1}(d)$, we have to find out which elements of $VTL_{n}(d)$ multiply the elements of $A_{n+1}$ become $\mathcal{CE}_{n+1}^{k}$.
In particular, when $k=n+1$, $\mathcal{CE}_{n+1}^{n+1}$ only comes from $1_n\cdot 1_{n+1}$.

In the following, we assume that $n+1-2[\frac{n+1}{2}]\leq k\leq n-1$ in the canonical $k$-element $\mathcal{CE}_{n+1}^{k}$.
For each $i\in \{k+1,k+3,...,n\}$, let $\mathcal{U}^{i}=\{U_{n}U_{n-1}\cdots U_{i+1}e_{i}|U_{j}=e_j ~\text{or}~ v_j, i+1\leq j\leq n\}$ for $i<n$ and $\mathcal{U}^{n}=\{e_{n}\}$ for $i=n$.
Note that the set $\mathcal{U}^{i}\subset A_{n+1}$.
Then each element of $\mathcal{U}^{i}$ is $(n-2)$-element.

\begin{lemma}\label{L:observations}
Let $\mathcal{CE}_{n+1}^{k}=XY$, where $X\in VTL_{n}(d)$ and $Y\in A_{n+1}$.
Then we have that
\begin{itemize}
  \item [(1)] In $Y$, there are a cap on some site $i$, $i\in\{k+1,k+3,...,n-2,n\}$, and there are $i-1$ vertical strands between first $i-1$ point pairs, that is, $Y\in \mathcal{U}^{i}$.
  \item [(2)] $X$ is a $(k+1)$-element, meanwhile, there are cups on sites $k+1$, $k+3$,...,$n-2$ and caps on sites $k+1$, $k+3$,...,$i-2$ in $X$ ($i\geq k+3$).
\end{itemize}
\end{lemma}
\begin{proof}
For conclusion $(1)$, first, we claim that there must have one cap in $Y$.
If not, then $Y$ does not contains any $e_i$ and $Y$ is a $n$-element. Then we get that $X 1_{n+1}=\mathcal{CE}_{n+1}^{k}Y^{-1}$. And $Y^{-1}$ will permute the bottom point site of $\mathcal{CE}_{n+1}^{k}$. Then there must be a cup on site $n$ in $X 1_{n+1}$, which is in contradiction with $X\in VTL_{n}(d)$.
Second, we claim that the cap connects continuous two points $i$ and $i+1$ in $Y$ ($i\leq n$).
If not, we assume the two points are $i$ and $j$ such that $j-i>1$.
For any $X$, there is still a cap connecting $i$ and $j$ such that $XY$ is impossible to be $\mathcal{CE}_{n+1}^{k}$.
Via the above two claims, it is obvious that there are $i-1$ vertical strands between first $i-1$ point pairs by the definition of the set $A_{n+1}$, that is, $Y\in \mathcal{U}^{i}$.

According conclusion (1), we can get conclusion (2) easily as shown in Figure \ref{F:observation}.
\end{proof}

\begin{figure}[!htbp]
\centering
\includegraphics[width=4.5in]{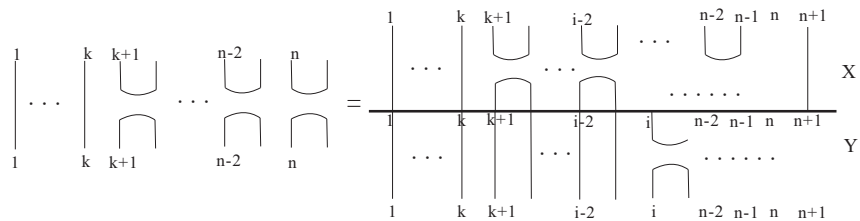}
  \caption{The illustrations in Lemma \ref{L:observations}.}\label{F:observation}
\end{figure}

Next we can get the following result.

\begin{lemma}\label{L:existandunique}
For each $Y\in \mathcal{U}^{i}$ ($i=k+1,k+3,...,n-2,n$), there exists unique $(k+1)$-element $X\in VTL_n(d)$ up to isotopy such that $\mathcal{CE}_{n+1}^{k}=XY$.
\end{lemma}
\begin{proof}
First, we use induction way to prove the existence.
In fact, the existence ensures the uniqueness based on the characterisation of the set $\mathcal{U}^{i}\subset A_{n+1}$ and $\mathcal{CE}_{n+1}^{k}$ .
The Lemma is clearly true when $i=n$, that is, we take $(k+1)$-element $X=e_{k+1}e_{k+3}\cdots e_{n-2}$ in $VTL_n(d)$.
Inductively suppose it is true for a given $k+3\leq i\leq n-2$.
Then for $i-2$, let $\mathcal{U}^{i-2}=\{U_{n}U_{n-1}\cdots U_{i}U_{i-1}e_{i-2}|U_{j}=e_j ~\text{or}~ v_j, i-1\leq j\leq n\}$.

\textbf{case 1.} $U_i=e_i$.

Note that for any $Y\in \mathcal{U}^{i-2}$, then $Y=(Ze_{i})U_{i-1}e_{i-2}$ where $Ze_i\in \mathcal{U}^{i}$.
By inductive assumption, there exists $(k+1)$-element $X^{*}\in VTL_n(d)$ such that $\mathcal{CE}_{n+1}^{k}=X^{*}(Ze_{i})$.
Then for $(Ze_{i})U_{i-1}e_{i-2}$, we take $(k+1)$-element $X=X^{*}$. Then
\begin{eqnarray*}
X^{*}(Ze_{i})U_{i-1}e_{i-2} &=& e_{k+1}e_{k+3}\cdots e_{i-2}e_{i}\cdots e_{n}U_{i-1}e_{i-2}\\
&=& e_{k+1}e_{k+3}\cdots (e_{i-2}e_{i}U_{i-1}e_{i-2})e_{i+2}\cdots e_{n}\\
&=&e_{k+1}e_{k+3}\cdots (e_{i}e_{i-2}U_{i-1}e_{i-2})e_{i+2}\cdots e_{n}\\
&=&e_{k+1}e_{k+3}\cdots (e_{i}(e_{i-2}U_{i-1}e_{i-2}))e_{i+2}\cdots e_{n}\\
\end{eqnarray*}
By Eq. (1), we have $e_{i-2}e_{i-1}e_{i-2}=e_{i-2}$ and $e_{i-2}v_{i-1}e_{i-2}=e_{i-2}$.
Then we have
\begin{eqnarray*}
X^{*}(Ze_{i})U_{i-1}e_{i-2}&=&e_{k+1}e_{k+3}\cdots (e_{i}e_{i-2})\cdots e_{n}\\
&=&e_{k+1}e_{k+3}\cdots e_{i-2}e_{i}\cdots e_{n}
\end{eqnarray*}

\textbf{case 2.} $U_i=v_i$.

Note that $Y=Wv_{i}U_{i-1}e_{i-2}$ where $W\in \{U_{n}U_{n-1}\cdots U_{i+1}|U_{j}=e_j ~\text{or}~ v_j, i+1\leq j\leq n\}$.

By inductive assumption, there exists $(k+1)$-element $X^{*}\in VTL_{n}(d)$ such that $\mathcal{CE}_{n+1}^{k}=X^{*}We_i$.

\textbf{subcase 2.1.} When $Y=Wv_{i}e_{i-1}e_{i-2}$, we take $(k+1)$-element $X=X^{*}v_{i-1}$.

Note that $X^{*}v_{i-1}Wv_{i}e_{i-1}e_{i-2}=X^{*}Wv_{i-1}v_{i}e_{i-1}e_{i-2}$ since the subscripts of $e$ and $v$ in $W$ are both at least $i+1$.
\begin{eqnarray*}
X^{*}Wv_{i-1}v_{i}e_{i-1}e_{i-2} &=& X^{*}W(v_{i-1}v_{i}e_{i-1})e_{i-2}\\
&=& X^{*}W(e_{i}e_{i-1})e_{i-2}\\
&=& e_{k+1}e_{k+3}\cdots e_{n}e_{i-1}e_{i-2}\\
&=& e_{k+1}e_{k+3}\cdots e_{i-4}e_{i}\cdots e_{n}e_{i-2}e_{i-1}e_{i-2}\\
&=& e_{k+1}e_{k+3}\cdots e_{i-4}e_{i}\cdots e_{n}e_{i-2}\\
&=& e_{k+1}e_{k+3}\cdots e_{i-4}e_{i-2}e_{i}\cdots e_{n}
\end{eqnarray*}

\textbf{subcase 2.2.}  When $Y=Wv_{i}v_{i-1}e_{i-2}$, we take $(k+1)$-element $X=X^{*}v_{i-1}v_{i-2}$.

Note that $X^{*}v_{i-1}v_{i-2}Wv_{i}v_{i-1}e_{i-2}=X^{*}Wv_{i-1}v_{i-2}v_{i}v_{i-1}e_{i-2}$ since the subscripts of $e$ and $v$ in $W$ are both at least $i+1$.
\begin{eqnarray*}
X^{*}Wv_{i-1}v_{i-2}v_{i}v_{i-1}e_{i-2} &=& X^{*}Wv_{i-1}v_{i}(v_{i-2}v_{i-1}e_{i-2})\\
&=& X^{*}Wv_{i-1}v_{i}(e_{i-1}e_{i-2})\\
&=& X^{*}W((v_{i-1}v_{i}e_{i-1})e_{i-2})\\
&=& X^{*}W((e_{i}e_{i-1})e_{i-2})\\
&=& e_{k+1}e_{k+3}\cdots e_{n}\ \ (using \ \ \textbf{subcase 2.1})
\end{eqnarray*}
\end{proof}

Next we can obtain the recursive formula for the coefficients of $\mathcal{CE}_{n+1}^{k}$ and $\mathcal{CE}_{n}^{k+1}$.

\begin{proposition}\label{P:recursiveformula}
The coefficient of $\mathcal{CE}_{n}^{k}$ in $f_{n}$ satisfies the following recursive formula:
\begin{equation}
  \coeff{n}{}(\mathcal{CE}_{n}^{k})=[\frac{n-k}{2}]\cdot\frac{-2}{n(d+2n-4)}\cdot\coeff{n-1}{}{([k+1]_{n-1})}.
\end{equation}
\end{proposition}

\begin{proof}
According to Lemma \ref{L:observations} and \ref{L:existandunique}, we can obtain that $XY$ becomes canonical $k$-element $\mathcal{CE}_{n+1}^{k}$ in $f_{n+1}$ if and only if $X$ is a $(k+1)$-element and $Y\in U^{i}$ for $i\in \{k+1,k+3,...,n\}$.

Thus the coefficient of $\mathcal{CE}_{n+1}^{k}$ is as follows.

\begin{equation}
\coeff{n+1}{}{(\mathcal{CE}_{n+1}^{k})}= \coeff{n}{}{([k+1]_n)}\sum_{i=n+2-2[\frac{n+1}{2}]}^{n} \sum _{Y\in \mathcal{U}^{i}}\coeff{n+1}{\mathcal{K}}(Y)
\label{E:CE}
\end{equation}

Recall that $\mathcal{U}^{i}=\{U_{n}U_{n-1}\cdots U_{i+1}e_{i}|U_{j}=e_j ~\text{or}~ v_j, i+1\leq j\leq n\}$ for $i<n$ and $\mathcal{U}^{n}=\{e_{n}\}$ for $i=n$.
First, we use induction way to prove $\sum _{Y\in \mathcal{U}^{i}}\coeff{n+1}{\mathcal{K}}(Y)=\frac{-2}{(n+1)(d+2n-2)}$, that is, $\sum _{Y\in \mathcal{U}^{i}}\coeff{n+1}{\mathcal{K}}(Y)$ is independent on $i$ ($1\leq i\leq n$).

(1) If $i=n$, then
\begin{equation*}
\sum _{Y\in \mathcal{U}^{n}}\coeff{n+1}{\mathcal{K}}(Y)=\coeff{n+1}{\mathcal{K}}(e_n)=y_nx_{n-1}
=\frac{-2n}{(n+1)(d+2n-2)}\cdot\frac{1}{n}=\frac{-2}{(n+1)(d+2n-2)}
\end{equation*}

(2) Inductively suppose it is true for a given $2\leq i\leq n$.
Then for $i-1$,

\begin{eqnarray*}
\sum _{Y\in \mathcal{U}^{i-1}}\coeff{n+1}{\mathcal{K}}(Y)&=&(\sum_{Y\in \mathcal{U}^{i}}\coeff{n+1}{\mathcal{K}}(Y))(1+\frac{z_{i}}{y_{i}})y_{i-1}\frac{x_{i-2}}{x_{i-1}}\\
&=&\frac{-2}{(n+1)(d+2n-2)}(1+\frac{d+2i-2}{-2})\frac{-2(i-1)}{i(d+2i-4)}\frac{i}{i-1}\\
&=&\frac{-2}{(n+1)(d+2n-2)}
\end{eqnarray*}

Based on $\sum _{Y\in \mathcal{U}^{i}}\coeff{n+1}{\mathcal{K}}(Y)=\frac{-2}{(n+1)(d+2n-2)}$, we can obtain that the left hand of Eq. (\ref{E:CE}) is equal to
$[\frac{n+1-k}{2}]\cdot\frac{-2}{(n+1)(d+2n-2)}\cdot\coeff{n}{}{([k+1]_n)}$.
\end{proof}

\begin{remark}
In \cite{Mor}, Morrison had used similar method to give a recursive formula for the coefficients in Jones-Wenzl projectors.
\end{remark}

\begin{corollary}\label{C:explctfml}
The explicit formula for $f_n$ in $VTL_n(d)$ is as follows:
\begin{equation}
f_n=\sum_{l=0}^{[\frac{n}{2}]}\frac{(-2)^{l}l!}{n!\prod_{i=1}^{l}(d+2n-2-2i)}[n-2l]_n.
\end{equation}
\end{corollary}
\begin{proof}
Recall that $[k]_n$ denotes the sum over all $k$-elements of $VTL_n(d)$.
$d$ denotes the evaluation value of a simple closed curve.
Let $n-2l=k$. By Proposition \ref{P:recursiveformula}, we have
\begin{eqnarray*}
\coeff{n}{}(\mathcal{CE}_{n}^{n-2l})&=&l\cdot\frac{-2}{n(d+2n-4)}\cdot\coeff{n-1}{}{([n-2l+1]_{n-1})}\\
  &=&l(l-1)\cdot\frac{(-2)^2}{n(n-1)(d+2n-4)(d+2n-6)}\cdot\coeff{n-2}{}{([n-2l+2]_{n-2})}\\
  &=&l!\cdot\frac{(-2)^l}{n(n-1)\cdots (n-l+1)\prod_{i=1}^{l}(d+2n-2-2i)}\cdot\coeff{n-l}{}{([n-l]_{n-l})}\\
  &=&l!\cdot\frac{(-2)^l}{n(n-1)\cdots (n-l+1)\prod_{i=1}^{l}(d+2n-2-2i)}\cdot\frac{1}{(n-l)!}\\
  &=&\frac{(-2)^{l}l!}{n!\prod_{i=1}^{l}(d+2n-2-2i)}
\end{eqnarray*}
\end{proof}

After calculating, we obtain that
\begin{equation*}
  f_2=\frac{1}{2}1_2-\frac{1}{d}e_1+\frac{1}{2}v_1=\frac{1}{2}[2]_2-\frac{1}{d}[0]_2,
\end{equation*}
\begin{equation*}
  f_3=\frac{1}{3!}[3]_3-\frac{2}{3!(d+2)}[1]_3,
\end{equation*}
\begin{equation*}
  f_4=\frac{1}{4!}[4]_4-\frac{2}{4!(d+4)}[2]_4+\frac{1}{3(d+2)(d+4)}[0]_4,
\end{equation*}
\begin{equation*}
  f_5=\frac{1}{5!}[5]_5-\frac{2}{5!(d+6)}[3]_5+\frac{1}{15(d+4)(d+6)}[1]_5,
\end{equation*}
\begin{equation*}
  f_6=\frac{1}{6!}[6]_6-\frac{2}{6!(d+8)}[4]_6+\frac{1}{90(d+6)(d+8)}[2]_6-\frac{1}{15(d+4)(d+6)(d+8)}[0]_6.
\end{equation*}
\begin{equation*}
  f_7=\frac{1}{7!}[7]_7-\frac{2}{7!(d+10)}[5]_7+\frac{1}{630(d+8)(d+10)}[3]_7-\frac{1}{105(d+6)(d+8)(d+10)}[1]_7.
\end{equation*}
\begin{eqnarray*}
  f_8&=&\frac{1}{8!}[8]_8-\frac{2}{8!(d+12)}[6]_8+\frac{1}{630\cdot8(d+10)(d+12)}[4]_8\\
& &-\frac{1}{105\cdot 8(d+8)(d+10)(d+12)}[2]_8+\frac{1}{105(d+6)(d+8)(d+10)(d+12)}[0]_8.
\end{eqnarray*}

\section{Some results on projectors}
The $\textit{Markov trace}$ on the virtual Temperley-Lieb algebra $VTL_{n}(d)$ defined in the following simple way which is similar with the definition of the Markov trace on the Temperley-Lieb algebra $TL_{n}(d)$.
If $D$ is a tangle diagram in the rectangle having $n$ points on each of the top and bottom edges of the rectangle, $D$ represents an element of $VTL_{n}(d).$
Then $tr(D)$ is the bracket polynomial, evaluated at the chosen value of $A$ in $\mathbb{C}$, of the link diagram formed from $D$ by jointing the points on the top edge of the rectangle to those on the bottom by arcs outside the rectangle that introduce no new crossing.
This idea (analogous to the closure of a braid) is illustrated in Figure \ref{F:closure}.

\begin{figure}
\center
\includegraphics{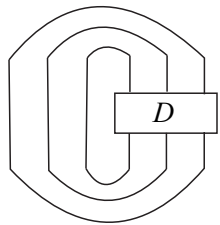}
\caption{The Markov trace of a tangle diagram $D$.}\label{F:closure}
\end{figure}

This clearly induces a well-defined linear map on the algebra $VTL_{n}(d)$, because the relations used to defined $VTL_{n}(d)$ are essentially the formulae that characterize the bracket polynomial. It is then clear that this trace function has the following properties.

\begin{align*}
  tr:& VTL_{n}(d)\rightarrow \mathbb{C},\ for \ any\ x,y\in VTL_{n}(d), \\
  &tr(xy)=tr(yx),\\
   &tr(1_{n})=d^n,\\
   &d\,tr(xe_n)=d\,tr(xv_n)=tr(x1_{n+1})=d\,tr(x)
\end{align*}

The following Lemma calculates the trace of the element $f_n\in VTL_{n}(d)$ constructed above.

\begin{lemma}
$tr(f_{n})=\frac{d^{n-1}(d+2n-2)\prod_{i=1}^{n-1}(d+i-2)}{n!}.$
\end{lemma}
\begin{proof}
\begin{eqnarray*}
tr(f_{n})&=&tr(x_{n-1}f_{n-1}+y_{n-1}f_{n-1}e_{n-1}f_{n-1}+z_{n-1}f_{n-1}v_{n-1}f_{n-1})\\
&=&x_{n-1}d tr(f_{n-1})+y_{n-1}tr(f_{n-1}f_{n-1}e_{n-1})+z_{n-1}tr(f_{n-1}f_{n-1}v_{n-1})\\
&=&(x_{n-1}d+y_{n-1}+z_{n-1})tr(f_{n-1})\\
&=&\alpha_{n-1}tr(f_{n-1})\\
&=&\prod_{i=1}^{n-1}\alpha_{i}tr(f_1)\\
&=&\frac{d^{n-1}(d+2n-2)\prod_{i=1}^{n-1}(d+i-2)}{n!}.
\end{eqnarray*}
\end{proof}

By a reduced word $w\in VTL_{n}(d)$ we shall mean a word in the set $\{1_{n}, e_1, e_2, ...,e_{i-1}, v_1,$ $v_2, ... , v_{i-1}\}$ that is not equal to $cw'$ for any $c$ a constant and $w'$ a word of smaller length. Using the relations of $VTL_{n}(d)$ and applying simple combinatorial arguments, we shall show the following result which is similar with the Jones-Wenzl projectors case \cite{AJL}.

\begin{proposition}\label{red}
A reduced word $w\in VTL_{n}(d)$ contains at most one term from the set $\{e_{n-1}$, $v_{n-1}\}$.
\end{proposition}
\begin{proof}
Based on the explicit formula of $f_n$ in Corollary \ref{C:explctfml} and the simplified recursive formula of $f_n$ in Lemma \ref{L:fmK}, it is obvious that $w$ appears in $f_n$ and $w$ can be expressed as the product of some element of $VTL_{n-1}(d)$ and another element of $A_{n}$.
Then we obtain that a reduced word $w\in VTL_{n}(d)$ contains at most one term from the set $\{e_{n-1}$, $v_{n-1}\}$ since any element of $A_{n}$ contains one term from the set $\{e_{n-1}$, $v_{n-1}\}$ and any element of $VTL_{n-1}(d)$ does not contain neither $e_{n-1}$ nor $v_{n-1}$.
\end{proof}

\section{Acknowledgements}
This work is supported partially by Hu Xiang Gao Ceng Ci Ren Cai Ju Jiao Gong Cheng-Chuang Xin Ren Cai (No. 2019RS1057).
Deng is also supported by Doctor's Funds of Xiangtan University (No. 09KZ$|$KZ08069) and NSFC (No. 12001464).
Jin is also supported by NSFC (No. 11671336) and the Fundamental Research
Funds for the Central Universities  (No. 20720190062).
Kauffman is also supported by the Laboratory of Topology and Dynamics, Novosibirsk State University (contract no. 14.Y26.31.0025 with the Ministry of Education and Science of the Russian Federation).

\end{document}